\documentclass[12pt]{article}
\usepackage{amsmath, amsthm}\usepackage{enumerate}\usepackage{amssymb}\usepackage{bbold}
\usepackage{bbm}\usepackage{dsfont}\usepackage{color}\usepackage{xcolor}\usepackage[explicit]{titlesec}
\newtheorem{theorem}{Theorem}[section]
\newtheorem{proposition}[theorem]{Proposition}

\newtheorem{lemma}[theorem]{Lemma}
\newtheorem{example}[theorem]{Example}
\newtheorem{corollary}[theorem]{Corollary}

\newtheorem{definition}[theorem]{Definition}

\makeatletter
\renewcommand{\subsection}{\@startsection{subsection}{1}
{0pt}{3.25ex plus 1ex minus.2ex}{-1em}{\normalfont\normalsize\bf}}
\makeatother
\begin{document}

\title{{\bf Duality and Norm Completeness in the Classes of
Limitedly-Lwc and Dunford--Pettis-Lwc Operators}}
\maketitle
\author{\centering{{Safak Alpay$^{1}$, Eduard Emelyanov$^{2}$, Svetlana Gorokhova $^{3}$\\ 
\small $1$ Middle East Technical University, Ankara, Turkey\\ 
\small $2$ Sobolev Institute of Mathematics, Novosibirsk, Russia\\
\small $3$ Uznyj Matematiceskij Institut VNC RAN, Vladikavkaz, Russia}

\abstract{We investigate the duality and norm completeness in the classes of 
limitedly L-weakly compact and Dunford--Pettis L-weakly compact  
operators from Banach spaces to Banach lattices.}

\vspace{5mm}


}}
\bigskip
\bigskip

\section{Introduction}

The theory of L-weakly compact (briefly, \text{\rm Lwc}) sets and operators was developed 
by P. Meyer-Nieberg in the beginning of seventies in order to diversify the concept 
of weakly compact operators via imposing the Banach 
lattice structure on the range of operators \cite{Mey0}. Dunford--Pettis sets appeared
a decade later in the work \cite{Andr} of K. T. Andrews. 
Shortly thereafter, J. Bourgain and J. Diestel introduced limited sets and operators \cite{BD}.
 Since then L-weakly compact operators and limited operators have attracted
permanent attention and inspiring researchers. Recently further related classes of operators were introduced 
and studied by many authors (see, for example, \cite{AEG,BLM1,EAS,EG,GM,Ghenciu2020,GE,BLM2,OM,Wnuk2013}, and references therein). 
Using the Meyer-Nieberg approach for the Dunford--Pettis and for limited
(instead of bounded) sets in the domain, we introduce 
Dunford--Pettis \text{\rm Lwc} and the limitedly \text{\rm Lwc} operators. 
We study the duality and norm completeness in classes of these operators.

Throughout the text: vector spaces are real; operators are linear and bounded; $X$ and $Y$ 
stand for Banach spaces, $E$ and $F$ for Banach lattices; $B_X$ for the closed unit ball of $X$;
$\text{\rm L}(X,Y)$ for the space of all bounded operators from $X$ to $Y$; 
$E_+$ for the positive cone of $E$;
$
  \text{\rm sol}(A):=\bigcup\limits_{a\in A}[-|a|,|a|]
$
for the solid hull of $A\subseteq E$; 
$E^a:=\{x\in E: |x|\ge x_n\downarrow 0\Rightarrow\|x_n\|\to 0\}$ 
for the order continuous part of $E$; and $X'$ for the norm-dual of $X$. 
We identify $X$ with its image $\widehat{X}$ in $X''$ 
under the canonical embedding $\hat{x}(f)=f(x)$.
For further terminology on Banach lattices, see \cite{AlBu,Kus,Mey}.

\medskip
The paper is organized as follows. 
In Section 2 we introduce Dun\-ford--Pettis \text{\rm Lwc} and
limitedly \text{\rm Lwc} opera\-tors, and investigate their basic properties. 
Among other things, we characterize Dun\-ford--Pettis sets via limited sets in
Theorem \ref{bi-limited=DP} and establish semi-duality for arbitrary 
limitedly \text{\rm Lwc} operators in Theorem \ref{MW--LW--duality}.
Section 3 is devoted to complete norms on spaces of such operators
in Banach lattice setting.

\section{Main Definitions and Basic Properties}

In this section we collect main definitions, introduce the Dunford--Pettis \text{\rm L}-weakly compact and limitedly 
\text{\rm L}-weakly compact operators from a Banach space to a Banach lattice,
and study their basic properties.
We begin with the following crucial definition belonging to P. Meyer-Nieberg \cite{Mey}.

\begin{definition}\label{LWC-subsets}
{\em A subset $A$ of $F$ is called an \text{\rm Lwc} {\em set} whenever
each disjoint sequence in $\text{\rm sol}(A)$ is norm-null.
A bounded operator $T:X\to F$ is an \text{\rm Lwc} {\em operator}
(briefly, $T\in\text{\rm Lwc}(X,F)$)  if $T(B_X)$ is an \text{\rm Lwc} subset of $F$.}
\end{definition}
\noindent
It can be easily seen that $B_E$ is not \text{\rm Lwc} unless
$\dim(E)<\infty$.
For every \text{\rm Lwc} subset $A$ of $E$, we have $A\subseteq E^a$. 
Indeed, otherwise there is $a\in A$
with $|a|\in E\setminus E^a$, and hence there exists a disjoint sequence 
$(x_n)$ in $[0,|a|]\subseteq\text{\rm sol}(A)$
with $\|x_n\|\not\to 0$. 
The next important fact goes back to Meyer-Nieberg (cf. \cite[Thm.5.63]{AlBu} 
and \cite[Prop.2.2]{BuDo} for more general setting).

\begin{proposition}\label{Burkinshaw--Dodds}
Let $A\subseteq E$ and $B\subseteq E'$ be nonempty bounded sets. Then
every disjoint sequence in $\text{\rm sol}(A)$ is uniformly null
on $B$ iff every disjoint sequence in $\text{\rm sol}(B)$ is uniformly null on $A$.
\end{proposition}
\noindent
We include a proof of the following certainly well known fact.

\begin{lemma}\label{just lemma}
Let $L$ be a nonempty bounded subset of $F'$. TFAE.
\begin{enumerate}[$i)$]
\item $L$ is an \text{\rm Lwc} subset of $F'$.
\item Each disjoint sequence in $B_F$  is uniformly null on $L$.
\item Each disjoint sequence in $B_{F''}$ is uniformly null on $L$.
\end{enumerate}
\end{lemma}

\begin{proof}
Since $\|f\|=\sup\{|f(x)|: x\in B_X\}=\sup\{|y(f)|: y\in B_{X''}\}$ then
$(f_n)\ \text{\rm in}\ X'\ \text{\rm converges uniformly on}\ B_X\ \text{\rm iff it converges uniformly on}\ B_{X''}$
under the identification of $f\in X'$ with $\hat{f}\in X'''$.
Now, applying Proposition \ref{Burkinshaw--Dodds} first to $A=B_F$ and $B=L$, 
and then to $A=L$ and $B=B_{F''}$ we obtain that both ii) and iii) are equivalent to the condition that 
every disjoint sequence in $\text{\rm sol}(L)$ is norm null,
which means that $L$ is an \text{\rm Lwc} subset of $F'$.
\end{proof}
\noindent
The following notions are due to K. Andrews, 
J. Bourgain, and J. Diestel.

\begin{definition}\label{bi-limited set}
{\em
A bounded subset $A$ of $X$ is called:
\begin{enumerate}[a)]
\item 
a {\em Dunford--Pettis set} (briefly, $A$ is \text{\rm DP})
if $(f_n)$ is uniformly null on $A$ for each 
\text{\rm w}-null $(f_n)$ in $X'$ (see \cite[Thm.1]{Andr}).
\item 
a {\em limited set} if $(f_n)$ is uniformly null on $A$
for each \text{\rm w}$^\ast$-null $(f_n)$ in $X'$ (see \cite{BD}).
\end{enumerate}}
\end{definition}

\noindent
In reflexive spaces, \text{\rm DP} sets and limited sets 
agree with relatively compact sets \cite{BD}. 
In general, 
$A \ \text{\rm is relatively compact} \ \Longrightarrow 
A \ \text{\rm is limited} \ \Longrightarrow A \ \text{\rm is DP}$. 
The unit ball $B_X$ is not limited in $X$ 
unless $\dim(X)<\infty$, and the limited sets 
are relatively compact in separable and in reflexive Banach spaces \cite{BD}. 
$B_{c_0}$ is not limited in $c_0$, yet $\widehat{B_{c_0}}$ is 
limited in $c''_0=\ell^\infty$ by Phillip's lemma. $B_{c_0}$ is \text{\rm DP} 
because $c'_0=\ell^1$ has the Schur property.
%
\noindent
The \text{\rm DP} sets turn to limited
while embedded in the bi-dual.

\begin{theorem}\label{bi-limited=DP}
{\em
Let $A\subseteq X$. TFAE:
\begin{enumerate}[$i)$]
\item
$A$ is a \text{\rm DP} subset of $X$.
\item
$\widehat{A}$ is limited in $X''$.
\end{enumerate}
}
\end{theorem}

\begin{proof}
$i)\Longrightarrow ii)$\
Assume $A$ is a \text{\rm DP} subset of $X$.
Let $f_n\stackrel{\text{\rm w}^\ast}{\to}0$ in $X'''$. 
Then $f_n|_X:=f_n|_{\widehat{X}}\stackrel{\text{\rm w}}{\to}0$ in $X'$, 
since $g(f_n|_X)=\hat{g}(\widehat{f_n|_X})=\hat{g}(f_n)=f_n(g)\to 0$
for each $g\in X''$. By the assumption, $(f_n|_X)$ is uniformly null on $A$,
and hence $(f_n)$ is uniformly null on $\widehat{A}$ as desired.
Therefore, $\widehat{A}$ is limited in $X''$.

$ii)\Longrightarrow i)$\ 
Suppose $\widehat{A}$ is limited in $X''$.
Let $f_n\stackrel{\text{\rm w}}{\to}0$ in $X'$. 
Then $\widehat{f_n}\stackrel{\text{\rm w}^\ast}{\to}0$ in $X'''$
and, as $\widehat{A}$ is limited, $(\widehat{f_n})$ is uniformly null on $\widehat{A}$.
Hence
$\sup_{a\in A}|f_n(a)|=\sup_{a\in A}|\hat{a}(f_n)|=
\sup_{a\in A}|\widehat{f_n}(\hat{a})|=\sup_{b\in \widehat{A}}|\widehat{f_n}(b)|\to 0$.
Therefore $(f_n)$ is uniformly null on $A$, which means that $A$ is \text{\rm DP} in $X$.
\end{proof}

\subsection{}
Recently K. Bouras, D. Lhaimer, and M. Moussa introduced and studied 
\text{\rm a-Lwc} {\em operators} from $X$ to $F$ carrying
weakly compact sets to \text{\rm Lwc} sets \cite{BLM1}. 
Here, we investigate operators carrying the Dunford--Pettis, or else limited
subsets of $X$ to \text{\rm Lwc} sets of $F$.
For the equivalence $i)\Longleftrightarrow ii)$ of
the following characterization of \text{\rm Lwc} sets, we refer 
to \cite[Prop.3.6.2]{Mey}, and the equivalence $i)\Longleftrightarrow iii)$ can be found in \cite[Lem.4.2]{BHM}

\begin{proposition}\label{Meyer 3.6.2}
For a nonempty bounded subset $A$ of $E$, TFAE.
\begin{enumerate}[$i)$]
\item 
$A$ is an \text{\rm Lwc} set.
\item 
For every $\varepsilon>0$, there exists $u_\varepsilon\in E_+^a$ such that
$A\subseteq[-u_\varepsilon,u_\varepsilon]+\varepsilon B_E$.
\item 
For every $\varepsilon>0$, there exists an \text{\rm Lwc} subset 
$A_\varepsilon$ of $E$ with 
$A\subseteq A_\varepsilon + \varepsilon B_E$.
\end{enumerate}
\end{proposition}

\noindent
A subset $A$ of $E$ is called {\em almost order bounded} whenever, for every $\varepsilon>0$, 
there is $u_\varepsilon\in E_+$ such that
$A\subseteq[-u_\varepsilon,u_\varepsilon]+\varepsilon B_E$.
Every relatively compact subset of $E$ is almost order bounded.
By Proposition \ref{Meyer 3.6.2}, each almost order bounded subset of $E^a$ is an \text{\rm Lwc} set. 

\begin{theorem}\label{l-LW-operators}
Let $T\in\text{\rm L}(X,F)$. TFAE.
\begin{enumerate}[$i)$]
\item 
$T$ takes limited subsets of $X$ onto \text{\rm Lwc} subsets of $F$.
\item 
$T$ takes compact subsets of $X$ onto \text{\rm Lwc} subsets of $F$.
\item 
$\{Tx\}$ is an \text{\rm Lwc} subset of $F$ for each $x\in X$.
\item 
$T(X)\subseteq F^a$.
\item 
$T'f_n\stackrel{\text{\rm w}^\ast}{\to}0$ in $X'$ for each disjoint bounded
sequence $(f_n)$ in $F'$.
\end{enumerate}
\end{theorem}

\begin{proof} 
The implications $i)\ \Longrightarrow\ ii)\ \Longrightarrow\ iii)$ are trivial, 
while $iii)\ \Longrightarrow\ iv)$ 
yields because each \text{\rm Lwc} subset of $F$ lies in $F^a$.

$iv)\ \Longrightarrow\ v)$: 
Let $(f_n)$ be a disjoint bounded sequence in $F'$ and $x\in X$. 
Since $T(X)\subseteq F^a$, $\{Tx\}$ is an \text{\rm Lwc} set, and hence 
$T'f_n(x)=f_n(Tx)\to 0$ by Proposition~\ref{Burkinshaw--Dodds}. As $x\in X$
is arbitrary, $(T'f_n)$ is $\text{\rm w}^\ast$-null.

$v)\ \Longrightarrow\ i)$: 
Assume in contrary $T(L)$ is not an \text{\rm Lwc} set in $F$ for some non-empty 
limited subset $L$ of $X$. By Proposition \ref{Burkinshaw--Dodds}, there exists 
a disjoint sequence
$(g_n)$ of $B_{F'}$ that is not uniformly null on $T(L)$.
Therefore $(T'g_n)$ is not uniformly null on $L$ violating 
$T'g_n\stackrel{\text{\rm w}^\ast}{\to}0$ and the limitedness of $L$.
The obtained contradiction completes the proof.
\end{proof}

\subsection{}
Because of Theorem \ref{l-LW-operators}\,$i)$, we prefer to call operators satisfying 
the equivalent conditions of Theorem \ref{l-LW-operators} by \text{\rm l-Lwc} {\em operators}
(they may equally deserve to be called {\em compactly} \text{\rm Lwc} {\em operators}
in view of Theorem \ref{l-LW-operators}\,$ii)$). 
While preparing the paper we
have learned that the operators satisfying of Theorem \ref{l-LW-operators}\,$v)$
have being already introduced and studied by F. Oughajji and M. Moussa 
under the name {\em weak \text{\rm L}-weakly compact operators} 
\cite{OM} (this name looks more suitable for \text{\rm a-Lwc} operators rather than for 
\text{\rm l-Lwc} operators).

\begin{definition}\label{Main LWC operators}
{\em An operator $T:X\to F$ is called:
\begin{enumerate}[a)]
\item 
a {\em Dunford--Pettis \text{\rm L}-weakly compact} (briefly, $T\in\text{\rm DP-Lwc}(X,F)$), 
if $T$ carries \text{\rm DP} subsets of $X$ onto \text{\rm Lwc} subsets of $F$.
\item 
{\em limitedly \text{\rm L}-weakly compact} (briefly, $T\in\text{\rm l-Lwc}(X,F)$), 
if $T$ carries limited subsets of $X$ onto \text{\rm Lwc} subsets of $F$.
\end{enumerate}}
\end{definition}
\noindent
Clearly, $\text{\rm DP-Lwc}(X,F)$ and $\text{\rm l-Lwc}(X,F)$
are vector spaces. Theorem \ref{l-LW-operators}\,$ii)$ implies
the second inclusion of the next formula, whereas the first one is trivial.
$$
   \text{\rm Lwc}(X,F)\subseteq\text{\rm a-Lwc}(X,F)\subseteq\text{\rm l-Lwc}(X,F).
   \eqno(1)
$$
\noindent
A Banach lattice $E$ has the {\em dual disjoint \text{\rm w}$^\ast$-property} 
(shortly, $E\in(\text{\rm DDw$^\ast$P})$) 
if each disjoint bounded sequence in $E'$ is w$^\ast$-null. 
We include the following corollary of Theorem \ref{l-LW-operators}.

\begin{corollary}\label{l-LW vs DDw*P}
{\em TFAE.
\begin{enumerate}[$i)$]
\item $F\in(\text{\rm DDw$^\ast$P})$.
\item $I_F\in\text{\rm l-Lwc}(F)$.
\item Each limited subset of $F$ is an $\text{\rm Lwc}$-set.
\item $\text{\rm l-Lwc}(F)=\text{\rm L}(F)$.
\item $\text{\rm l-Lwc}(X,F)=\text{\rm L}(X,F)$ for each Banach space $X$.
\end{enumerate}}
\end{corollary}
\begin{proof}
The equivalence $i) \Longleftrightarrow ii)$ follows from Theorem \ref{l-LW-operators}.

The implications $v)\Longrightarrow iv) \Longrightarrow ii) \Longleftrightarrow iii)$ are trivial.

$iii)\Longrightarrow v)$:\
Let $T\in\text{\rm L}(X,F)$ and let
$L$ be limited subset of $X$. Then $T(L)$ is a limited subset of $F$, and hence 
$T(L)$ is an \text{\rm Lwc} subset of $F$. Thus, $T\in\text{\rm l-Lwc}(X,F)$,
as desired.
\end{proof}
\noindent
The following version of Theorem \ref{l-LW-operators} is quite useful.

\begin{theorem}\label{l-LW-operators in dual}
Let $T\in\text{\rm L}(X,F')$. TFAE.
\begin{enumerate}[$i)$]
\item 
$T\in\text{\rm l-Lwc}(X,F')$.
\item 
$T$ takes compact subsets of $X$ onto \text{\rm Lwc} subsets of $F'$.
\item 
$\{Tx\}$ is an \text{\rm Lwc} subset of $F'$ for each $x\in X$.
\item 
$T(X)\subseteq (F')^a$.
\item 
$T'f_n\stackrel{\text{\rm w}^\ast}{\to}0$ in $X'$ for each disjoint bounded
sequence $(f_n)$ in $F''$.
\item 
$T'\widehat{g_n}\stackrel{\text{\rm w}^\ast}{\to}0$ in $X'$ for each disjoint bounded
sequence $(g_n)$ in $F$.
\end{enumerate}
\end{theorem}

\begin{proof} 
The equivalence $i)\ \Longleftrightarrow\ ii)\ \Longleftrightarrow\ iii)\ 
\Longleftrightarrow\ iv)\ \Longleftrightarrow\ v)$ 
follows from Theorem \ref{l-LW-operators}, and the implication $v)\ \Longrightarrow\ vi)$ is trivial.

$vi)\ \Longrightarrow\ i)$: 
Assume in contrary $T(L)$ is not an $\text{\rm Lwc}$-set in $F'$ for some non-empty 
limited subset $L$ of $X$. By Lemma \ref{just lemma}, there exists 
a disjoint sequence $(g_n)$ of $B_{F}$ such that $(\widehat{g_n})$
is not uniformly null on $T(L)$.
Therefore $(T'\widehat{g_n})$ is not uniformly null on $L$, which is absurd because of 
$T'\widehat{g_n}\stackrel{\text{\rm w}^\ast}{\to}0$ in $X'$ 
and $L$ is limited in $X$.
The obtained contradiction completes the proof.
\end{proof}

\subsection{}
Following G. Emmanuele \cite{Emma}, a Banach space $X$ is said to possess 
the {\em Bourgain--Diestel property} 
if each limited subset of $X$ is relatively weakly compact, and 
an operator $T:X\to Y$ is called a {\em Bourgain--Diestel operator} 
(briefly, $T\in\text{\rm BD}(X,Y)$) if $T$ carries limited sets onto relatively weakly compact sets.
The weakly compactness of \text{\rm Lwc} sets, Definitions \ref{bi-limited set}, \ref{Main LWC operators}, 
and  Theorem~\ref{l-LW-operators} together imply
$$  
   \text{\rm Lwc}(X,F)\subseteq
   \text{\rm DP-Lwc}(X,F)\subseteq\text{\rm l-Lwc}(X,F)\subseteq\text{\rm BD}(X,F).
   \eqno(2)
$$ 
All inclusions in (2) are generally proper by items d)--f) of Example \ref{Main example}.

\begin{example}\label{Main example}
{\em
\begin{enumerate}[a)]
\item 
$\text{\rm Id}_{\ell^1}\in\text{\rm a-Lwc}(\ell^1)\setminus\text{\rm Lwc}(\ell^1)$
(see \cite[p.1435]{BLM1}).
\item 
$\text{\rm Id}_{\ell^2}\in\text{\rm l-Lwc}(\ell^2)\setminus\text{\rm a-Lwc}(\ell^2)$
since limited sets in $\ell^2$ coincide with relatively compact sets that are in turn 
\text{\rm l-Lwc} sets in $\ell^2$, while $B_{\ell^2}$ is weakly compact but not 
an \text{\rm l-Lwc} set.
\item 
It is easy to see that 
$$
  T:=\text{\rm Id}_{c_0}\in\text{\rm l-Lwc}(c_0),\ \text{\rm yet} \
  T''=\text{\rm Id}_{c_0}''=\text{\rm Id}_{\ell^\infty}\notin
  \text{\rm l-Lwc}(\ell^\infty)=\text{\rm l-Lwc}(c''_0). 
$$
\item 
Since $\text{\rm l-Lwc}(\ell^2)=\text{\rm DP-Lwc}(\ell^2)$ due to reflexivity of $\ell^2$,
item b) implies $\text{\rm Id}_{\ell^2}\in\text{\rm DP-Lwc}(\ell^2)\setminus\text{\rm a-Lwc}(\ell^2)$.
We have no example of an operator
$T\in\text{\rm a-Lwc}(X,F)\setminus\text{\rm DP-Lwc}(X,F)$.
\item
$\text{\rm Id}_{c_0}\in\text{\rm l-Lwc}(c_0)\setminus\text{\rm DP-Lwc}(c_0)$
as $B_{c_0}$ is not \text{\rm Lwc} yet is a \text{\rm DP} set in $c_0$.
\item 
$\text{\rm Id}_{c}\in\text{\rm BD}(c)\setminus\text{\rm l-Lwc}(c)$ 
since limited sets in $c$ coincide with relatively compact sets, while
$c^a=c_0\subsetneqq c$ implies $\text{\rm Id}_{c}\notin\text{\rm l-Lwc}(c)$
by Theorem \ref{l-LW-operators}.
\item 
Combining examples d)--f) in one diagonal operator $(3\times 3)$-matrix: 
$$
   \text{\rm Lwc}(\ell^2\oplus c_0\oplus c)\subsetneqq
   \text{\rm DP-Lwc}(\ell^2\oplus c_0\oplus c)\subsetneqq
   \text{\rm l-Lwc}(\ell^2\oplus c_0\oplus c)\subsetneqq
   \text{\rm BD}(\ell^2\oplus c_0\oplus c).
$$
\end{enumerate}}
\end{example}

\subsection{}
The next result is a consequence of Theorem \ref{l-LW-operators}.

\begin{corollary}\label{DP--LW-operators}
Let $T\in\text{\rm L}(X,F)$. The following four conditions are equivalent:
\begin{enumerate}[$i)$]
\item 
$T''\in\text{\rm l-Lwc}(X'',F'')$.
\item 
$T''$ takes compact subsets of $X''$ to \text{\rm Lwc} subsets of $F''$.
\item 
$T''(X'')\subseteq(F'')^a$.
\item 
$T'''f_n\stackrel{\text{\rm w}^\ast}{\to}0$ in $X'''$ for each disjoint bounded
sequence $(f_n)$ in $F'''$.
\end{enumerate}
Each of above equivalent conditions implies:
\begin{enumerate}[]
\item[$v)$] 
$T\in\text{\rm DP-Lwc}(X,F)$.
\end{enumerate}
The condition $v)$ in turn implies:
\begin{enumerate}[]
\item[$vi)$] 
$T\in\text{\rm l-Lwc}(X,F)$.
\end{enumerate}
\end{corollary}

\begin{proof} 
Theorem \ref{l-LW-operators} implies 
$i)\ \Longleftrightarrow\ ii) \Longleftrightarrow\ iii) \Longleftrightarrow\ iv)$. 

$i)\ \Longrightarrow\ v)$:\ 
Let $A$ be a \text{\rm DP} subset of $X$. By Theorem \ref{bi-limited=DP},
$\widehat{A}$ is limited in $X''$. Then $\widehat{T(A)}=T''(\widehat{A})$ is an \text{\rm Lwc} subset 
of $F''$, and hence $T(A)$ is an \text{\rm Lwc} subset of $F$.
The latter means $T\in\text{\rm DP-Lwc}(X,F)$.

The implication $v)\ \Longrightarrow\ vi)$\ is trivial.
\end{proof}
\noindent
Note that, $T\in\text{\rm l-Lwc}(X,F)$ does not imply $T''\in\text{\rm l-Lwc}(X'',F'')$
in general (see Example \ref{Main example}~c)).
If $T''\in\text{\rm DP-Lwc}(X'',F'')$ then $T''\in\text{\rm l-Lwc}(X'',F'')$ by (2),
and hence $T\in\text{\rm DP-Lwc}(X,F)$ by Corollary \ref{DP--LW-operators}.
We have no example of an operator $T\in\text{\rm DP-Lwc}(X,F)$ such that 
$T''\notin\text{\rm DP-Lwc}(X'',F'')$.

\subsection{}
The following definition \cite[Def.3.1]{OM} was taken a starting point in~\cite{OM}. 
In our approach, this definition is a derivation of Theorem \ref{l-LW-operators}\,$iv)$, similarly
to the classical approach to \text{\rm Mwc} operators, which were introduced in \cite{Mey0} as
a derivation of \text{\rm Lwc} operators.

\begin{definition}\label{Main MWC operators}
{\em An operator $T:E\to Y$ is {\em limitedly \text{\rm Mwc}}
(briefly, $T\in\text{\rm l-Mwc}(E,Y)$), if $Tx_n\stackrel{\text{\rm w}}{\to}0$ for 
every disjoint bounded sequence $(x_n)$ in $E$.}
\end{definition}
\noindent
Recall that $E'$ is a KB space iff every disjoint bounded sequence in 
$E$ is \text{\rm w}-null (cf. \cite[Thm.4.59]{AlBu} and \cite[2.4.14]{Mey}).
Consequently $I_E$ is \text{\rm l-Mwc} iff $E'$ is a KB-space.
Now, we discuss the semi-duality for \text{\rm l-Lwc} and \text{\rm l-Mwc} operators. 
It was proved in \cite{OM} that 
$T\in\text{\rm l-Mwc}_+(E,F)$ iff $T'\in\text{\rm l-Lwc}_+(F',E')$. The next theorem
is an extension of \cite[Thm.4.13]{OM}, where only the case of positive operators is considered.

\begin{theorem}\label{MW--LW--duality}
The following statements hold:
\begin{enumerate}[$i)$]
\item 
$T'\in\text{\rm l-Mwc}(F',X')\Longrightarrow T\in\text{\rm l-Lwc}(X,F)$.
\item 
$T'\in\text{\rm l-Lwc}(Y',E')\Longleftrightarrow T\in\text{\rm l-Mwc}(E,Y)$.
\end{enumerate}
\end{theorem}

\begin{proof} 
$i)$\ 
Let $T'\in\text{\rm l-Mwc}(F',X')$, and let $(f_n)$ be disjoint bounded in $F'$.
Then $T'f_n\stackrel{\text{\rm w}}{\to}0$, and hence $T'f_n\stackrel{\text{\rm w}^\ast}{\to}0$. 
Theorem \ref{l-LW-operators} implies $T\in\text{\rm l-Lwc}(X,F)$.

$ii)$\ 
$(\Longleftarrow)$:\
Let $T\in\text{\rm l-Mwc}(E,Y)$. By Theorem \ref{l-LW-operators}, 
to show $T'\in\text{\rm l-Lwc}(Y',E')$, we need to prove that
$\{T'f\}$ is an \text{\rm l-Lwc} subset of $E'$ for each $f\in Y'$.
Let $f\in Y'$.
By Lemma \ref{just lemma}, it suffices to show $f(Tx_n)\to 0$
for each disjoint sequence $(x_n)$ in $B_E$.
So, let $(x_n)$ be disjoint in $B_E$.
Since $T\in\text{\rm l-Mwc}(E,Y)$ then $Tx_n\stackrel{\text{\rm w}}{\to}0$,
and hence $f(Tx_n)\to 0$, as desired.

$(\Longrightarrow)$:\
Let $T'\in\text{\rm l-Lwc}(Y',E')$.  Then 
$\{T'g\}$ is an \text{\rm Lwc} subset of $E'$ for each $g\in Y'$
by Theorem \ref{MW--LW--duality}. It follows from Lemma \ref{just lemma} that
$g(Tx_n)=T'g(x_n)\to 0$ for each disjoint bounded sequence $(x_n)$ in $E$.
Since $g\in Y'$ is arbitrary, $Tx_n\stackrel{\text{\rm w}}{\to}0$ for 
every disjoint bounded $(x_n)$ in $E$, and therefore $T\in\text{\rm l-Mwc}(E,Y)$.
\end{proof} 

\noindent
Note that the similar semi-duality was established in \cite[Thm.2.5]{BLM1}
for almost \text{\rm L}-weakly compact operators, and in \cite[Thm.2.3]{BLM2}
for order \text{\rm L}-weakly compact operators.

\subsection{}
Although, we have no sequential 
characterization of \text{\rm DP-Lwc} operators like the characterization 
of \text{\rm l-Lwc} operators given by Theorem \ref{l-LW-operators}\,$v)$, we
include the following result in this direction.

\begin{theorem}\label{DP-LW-semi-description}
Let $T\in\text{\rm L}(X,F)$. TFAE. 
\begin{enumerate}[$i)$]
\item 
$T''\in\text{\rm DP-Lwc}(X'',F'')$.
\item 
$T'f_n\stackrel{\text{\rm w}}{\to}0$ in $X'$ for each disjoint $(f_n)$ in $B_{F'}$.
\end{enumerate}
\end{theorem}

\begin{proof} 
$i)\ \Longrightarrow\ ii)$:\ 
Let $(f_n)$ be a disjoint sequence in $B_{F'}$. Suppose
$(T'f_n)$ is not \text{\rm w}-null in $X'$. 
Then $\widehat{f_n}(T''g)=T''g(f_n)=g(T'f_n)\not\to 0$ for some $g\in X''$,
and hence $(\widehat{f_n})$ is not uniformly null on $\{T''g\}$.
Lemma \ref{just lemma} implies that $\{T''g\}$ is not an \text{\rm Lwc} subset of $F''$.
However, $\{g\}$ is a \text{\rm DP} subset of $X''$ and then
$\{T''g\}$ must be \text{\rm Lwc} in $F''$ by the condition $i)$.
The obtained contradiction proves the implication.

$ii)\ \Longrightarrow\ i)$:\ 
Suppose in contrary $T''\notin\text{\rm DP-Lwc}(X'',F'')$.
Then $T''(A)$ is not \text{\rm Lwc} in $F''$ for some \text{\rm DP} subset $A$ of $X''$,
and hence $(\widehat{f_n})$ is not uniformly null on $T''(A)$ for some
disjoint sequence $(f_n)$ in $B_{F'}$, by Lemma~\ref{just lemma}.
Thus $(\widehat{T'f_n})=(T'''\widehat{f_n})$ is not uniformly null on $A$.
By $ii)$, $T'f_n\stackrel{\text{\rm w}}{\to}0$ in $X'$, and hence
$\widehat{T'f_n}\stackrel{\text{\rm w}}{\to}0$ in $X'''$. 
Since $A$ is a \text{\rm DP} subset of $X''$ then 
$(\widehat{T'f_n})$ is uniformly null on $A$.
The obtained contradiction completes the proof.
\end{proof}

\subsection{}
Clearly, $\text{\rm DP-Lwc}(X,F)$, $\text{\rm l-Lwc}(X,F)$, 
and $\text{\rm l-Mwc}(E,Y)$ are vector spaces.
It is natural to ask whether or not $\text{\rm DP-Lwc}(X,F)$,
$\text{\rm l-Lwc}(X,F)$, and $\text{\rm l-Mwc}(E,Y)$ are Banach spaces
under the operator norm. The answer to the question is positive. The details follow.

\begin{proposition}\label{DP-LW-closed}
Let\ $\text{\rm DP-Lwc}(X,F)\ni T_n\stackrel{\|\cdot\|}{\to} T$. 
Then $T\in\text{\rm DP-Lwc}(X,F)$.
\end{proposition}
\begin{proof}
Let $D$ be a \text{\rm DP} subset of $X$. WLOG $D\subseteq B_X$.
Take an arbitrary $\varepsilon>0$ and pick $k\in\mathbb{N}$ 
with $\|T-T_k\|\le\varepsilon$.
Since $T_k\in\text{\rm DP-Lwc}(X,F)$ then $T_k(D)$ is
an \text{\rm Lwc} subset of $F$.
As $T(D)\subseteq T_k(D)+\varepsilon B_F$,
Proposition \ref{Meyer 3.6.2} implies that $T(D)$ is
an \text{\rm Lwc} subset of $F$, hence $T\in\text{\rm DP-Lwc}(X,F)$.
\end{proof}
\noindent
Although the next proposition has a proof similar to 
the proof of Proposition \ref{DP-LW-closed}, we give an alternative proof.

\begin{proposition}\label{l-LW-closed}
Let\ $\text{\rm l-Lwc}(X,F)\ni T_n\stackrel{\|\cdot\|}{\to} T$. Then $T\in\text{\rm l-Lwc}(X,F)$.
\end{proposition}
\begin{proof}
Let $(f_n)$ be disjoint bounded in $F'$, and $x\in X$. By Theorem \ref{l-LW-operators},
we need to show $T'f_n(x)\to 0$. Let $\varepsilon>0$. Pick any $k\in\mathbb{N}$
with $\|T-T_k\|\le\varepsilon$. Since $T_k\in\text{\rm l-Lwc}(X,F)$  
then $|T_k'f_n(x)|\le\varepsilon$ for $n\ge n_0$. As $\varepsilon>0$ is arbitrary,
it follows from
$$
   |T'f_n(x)|\le|T'f_n(x)-T'_kf_n(x)|+|T_k'f_n(x)|\le 
$$
$$
   \|T'-T'_k\|\|f_n\|\|x\|+|T_k'f_n(x)|\le(\|f_n\|\|x\|+1)\varepsilon 
   \ \ \ \ (\forall n\ge n_0)
$$
that $T'f_n(x)\to 0$, as desired.
\end{proof}

\begin{proposition}\label{l-MW-closed}
Let $\text{\rm l-Mwc}(E,Y)\ni T_n\stackrel{\|\cdot\|}{\to} T$. Then $T\in\text{\rm l-Mwc}(E,Y)$.
\end{proposition}

\begin{proof}
By Theorem \ref{MW--LW--duality}\,$ii)$, 
$\text{\rm l-Lwc}(Y',E')\ni T'_n\stackrel{\|\cdot\|}{\to}T'$. By
Proposition \ref{l-LW-closed}, we have $T'\in\text{\rm l-Lwc}(Y',E')$.
Then $T\in\text{\rm l-Mwc}(E,Y)$ by Theorem \ref{MW--LW--duality}.
\end{proof}

\begin{corollary}\label{l-LW-algebra}
Let $E$ be a Banach lattice. The following holds.
\begin{enumerate}[$i)$]
\item 
$\text{\rm l-Lwc}(E)$ is a closed right ideal in $\text{\rm L}(E)$ 
$($and hence a subalgebra of $\text{\rm L}(E))$, and it is unital iff 
$I_E$ is \text{\rm l-Lwc}.
\item 
$\text{\rm l-Mwc}(E)$ is a closed left ideal in $\text{\rm L}(E)$ 
$($and hence a subalgebra of $\text{\rm L}(E))$, and it is unital iff 
$I_E$ is \text{\rm l-Mwc}.
\end{enumerate}
\end{corollary}

\begin{proof}
$i)$\
$\text{\rm l-Lwc}(E)$ is a closed subspace of $\text{\rm L}(E)$ by 
Proposition \ref{l-LW-closed}. 
As bounded operators carry limited sets onto limited sets, $\text{\rm l-Lwc}(E)$ 
is a right ideal in $\text{\rm L}(E)$.
The condition on $I_E$ making algebra $\text{\rm l-Lwc}(E)$ unital is trivial.

$ii)$\
By Proposition \ref{l-MW-closed}, $\text{\rm l-Mwc}(E)$ is a closed subspace of $\text{\rm L}(E)$.
It remains to show that $\text{\rm l-Mwc}(E)$ is a left ideal in $\text{\rm L}(E)$.
Let $T\in\text{\rm l-Mwc}(E)$ and $S\in\text{\rm L}(E)$. Then $T'\in\text{\rm l-Lwc}(E')$,
and hence, $i)$ implies $(ST)'=T'S'\in\text{\rm l-Lwc}(E')$.
Now, $ST\in\text{\rm l-Mwc}(E)$ by Theorem \ref{MW--LW--duality}.
\end{proof}

\section{The Banach Lattice Case}

Here, in the Banach lattice setting, we investigate the completeness 
of linear spans of positive operators introduced in Section 2 in the regular norm.
We begin with some technical notions.
Let $\emptyset\ne{\cal P}\subseteq\text{\rm L}(E,F)$. 
The set ${\cal P}$ is said to satisfy the {\em domination property} if
$0\le S\le T\in {\cal P} \ \Longrightarrow  \  S\in {\cal P}$. 
We say that  $T\in\text{\rm L}(E,F)$ is 
${\cal P}$-{\em dominated}, if $\pm T\le U\in{\cal P}$. 
An operator $T:E\to F$ is called an \text{\rm r}-${\cal P}$-{\it operator} 
if $T=T_1-T_2$, where $T_1,T_2$ are positive operators of~${\cal P}$.
We denote the collection of all \text{\rm r}-${\cal P}$-operators from 
$E$ to $F$ by $\text{\rm r-}{\cal P}(E,F)$.

\begin{proposition}\label{prop elem}
Let ${\cal P}\subseteq\text{\rm L}(E,F)$, ${\cal P}\pm{\cal P}\subseteq{\cal P}\ne\emptyset$, 
and $T\in\text{\rm L}(E,F)$. Then the following holds.

$i)$ \  $T\in\text{\rm r-}{\cal P}(E,F)\Longleftrightarrow 
T \ \text{is a}\  {\cal P}\text{-dominated\ operator\ from} \  {\cal P}$.

$ii)$ Assume the modulus $|T|$ of $T$ exists in $\text{\rm L}(E,F)$ and  
${\cal P}$ possesses the domination property. Then
$T\in\text{\rm r-}{\cal P}(E,F)\Longleftrightarrow |T| \in {\cal P}$.
\end{proposition}

\begin{proof}
$i)$ Let $T=T_1-T_2$, where  $T_1,T_2\in {\cal P}$ are positive.
${\cal P}\pm{\cal P}\subseteq{\cal P}$ implies $T\in {\cal P}$ and
$U=T_1+T_2\in {\cal P}$. From $\pm T\le U$, we obtain that $T$ is ${\cal P}$-dominated.

Now, let $T\in{\cal P}$ be ${\cal P}$-dominated. 
Take $U\in {\cal P}$ such that $\pm T\le U$.
Since $T=U-(U-T)$, and both $U$ and $U-T$ are 
positive operators in ${\cal P}$, $T$ is an \text{\rm r}-${\cal P}$-operator.

$ii)$ First assume $|T|\in {\cal P}$. 
Since $T=T_+-T_-$, $0\le T_{\pm}\le|T|\in {\cal P}$, the domination property 
implies that $T_+$ and $T_-$ are positive operators in ${\cal P}$, and hence $T=T_+-T_-$ 
is an \text{\rm r}-${\cal P}$-operator.

Now, assume $T$ is an \text{\rm r}-${\cal P}$-operator. 
Then there are positive $T_1,T_2\in{\cal P}$ 
satisfying $T=T_1-T_2$. Since $0\le T_+\le T_1$ and 
$0\le T_-\le T_2$, the domination property implies $T_+,T_-\in{\cal P}$. 
Hence $|T|=T_++T_-\in{\cal P}$.
\end{proof}

\begin{proposition}\label{vect lat}
Let $F$ be Dedekind complete, and let ${\cal P}$ be
a subspace in $\text{\rm L}(E,F)$ possessing the domination property. 
Then $\text{\rm r-}{\cal P}(E,F)$
is an order ideal in the Dedekind complete vector lattice
$\text{\rm L}_r(E,F)$ of regular operators from $E$ to $F$.
\end{proposition}

\begin{proof}
Since $F$ is Dedekind complete, $\text{\rm L}_r(E,F)$
is a Dedekind complete vector lattice.
By Proposition~\ref{prop elem}\,$ii)$, 
$T\in\text{\rm r-}{\cal P}(E,F)\Longrightarrow |T|\in\text{\rm r-}{\cal P}(E,F)$,
and hence $\text{\rm r-}{\cal P}(E,F)$ is a vector sublattice of $\text{\rm L}_r(E,F)$.
Since ${\cal P}$ has the domination property, $\text{\rm r-}{\cal P}(E,F)$
is an order ideal in $\text{\rm L}_r(E,F)$.
\end{proof}

\begin{lemma}\label{P-norm}{\em
Let ${\cal P}$ be a closed in the operator norm subspace 
of $\text{\rm L}(E)$. Then
$$
   \|T\|_{\text{\rm r-}{\cal P}}=\inf\{\|S\|:\pm T\le S\in{\cal P}\}
$$
defines a norm on the space ${\text{\rm r-}{\cal P}}(E)$ of all
${\text{\rm r-}{\cal P}}$-operators in $E$. Furthermore,
$\|T\|_{\text{\rm r-}{\cal P}}\ge\|T\|_r\ge\|T\|$ (where  
$\|T\|_r:=\inf\{\|S\|:\pm T\le S\in\text{\rm L}(E)\}$ is the regular norm)
for all $T\in{\text{\rm r-}{\cal P}}(E)$, and  
$({\text{\rm r-}{\cal P}},\|\cdot\|_{\text{\rm r-}{\cal P}})$ \ is a Banach space.}
\end{lemma}

\begin{proof}
It should be clear that $\|\cdot\|_{\text{\rm r-}{\cal P}}$ is a norm 
satisfying $\|\cdot\|_{\text{\rm r-}{\cal P}}\ge\|\cdot\|_r\ge\|\cdot\|$. 
Take a sequence $(T_n)$ in ${\text{\rm r-}{\cal P}}(E)$ that is Cauchy
in $\|\cdot\|_{\text{\rm r-}{\cal P}}$,
say $T_n=G_n-R_n$ for some $G_n,R_n\in{\cal P}_+$.
WLOG, $\|T_{n+1}-T_n\|_{\text{\rm r-}{\cal P}}<2^{-n}$ for all
$n\in\mathbb{N}$. As $\|\cdot\|_{\text{\rm r-}{\cal P}}\ge\|\cdot\|$, there exists
$T\in\text{\rm L}(E)$ such that $\|T-T_n\|\to 0$.
Since ${\cal P}$ is closed in the operator norm, $T\in{\cal P}$.
Pick $S_n\in{\cal P}$ with $\|S_n\|<2^{-n}$ and $\pm(T_{n+1}-T_n)\le S_n$. Then 
$$
   T_{n+1}x^+-T_nx^+\le Sx^+ \ \ \text{and} \ \ 
   -T_{n+1}x^-+T_nx^-\le Sx^- 
   \eqno(3) 
$$
for each $x\in E$. Summing up the inequalities in (3) gives 
$T_{n+1}x-T_nx\le S_n|x|$. Replacing $x$ by $-x$ 
gives $T_nx-T_{n+1}x\le S_n|x|$, and hence
$$
   |(T_{n+1}-T_n)x|\le S_n|x| \ \ \ (\forall x\in E).
   \eqno(4) 
$$
As ${\cal P}$ is closed in the operator norm, 
$Q_n:=\sum\limits_{k=n}^\infty S_k\in{\cal P}$ for all $n$. By (4), 
$$
   |(T-T_n)x|=\lim\limits_{k\to\infty}|(T_k-T_n)x|\le
   \sum\limits_{k=n}^\infty|(T_{k+1}-T_n)x|\le Q_n|x| \quad (x\in E), 
$$
and hence $\pm(T-T_n)\le Q_n$. Then
$$
  -Q_n\le(T-T_n)\le Q_n\ \text{\rm and} \ 0\le(T-T_n)+Q_n
$$
for all $n\in\mathbb{N}$. Therefore 
$$
   T=[(T-T_n)+Q_n]+[T_n-Q_n]=
   [(T-T_n)+Q_n+G_n]-[R_n+Q_n]\in{\text{\rm r-}{\cal P}}(E),
$$
and hence $(T-T_n)\in{\text{\rm r-}{\cal P}}(E)$ for all $n\in\mathbb{N}$.
As $\|T-T_n\|_{\text{\rm r-}{\cal P}}\le\|Q_n\|<2^{1-n}$, we conclude 
$T_n\stackrel{\|\cdot\|_{{\text{\rm r-}{\cal P}}}}{\longrightarrow}T$.
\end{proof}
\noindent
The following proposition is contained in \cite[Thm.3.6 and Thm.4.6]{OM}. 

\begin{proposition}\label{l-LW-domination}
Let operators $S,T\in\text{\rm L}(E,F)$ satisfy $\pm S\le T$. Then:
\begin{enumerate}[$i)$]
\item
$T\in\text{\rm l-Lwc}(E,F)\Longrightarrow S\in\text{\rm l-Lwc}(E,F)$. 
\item
$T\in\text{\rm l-Mwc}(E,F)\Longrightarrow S\in\text{\rm l-Mwc}(E,F)$. 
\end{enumerate}
\end{proposition}
\noindent
We include here the following partial result on domination for \text{\rm DP-Lwc}
operators.

\begin{proposition}\label{DP-LW-domination}
Let operators $S,T\in\text{\rm L}(E,F)$ satisfy $0\le S\le T$. If
$T''\in\text{\rm DP-Lwc}(E'',F'')$ then $S''\in\text{\rm DP-Lwc}(E'',F'')$.
\end{proposition}

\begin{proof}
Let $(f_n)$ be a disjoint sequence in $B_{F'}$. Then $(|f_n|)$ is also disjoint in $B_{F'}$,
and hence $T'|f_n|\stackrel{\text{\rm w}}{\to}0$ by Theorem \ref{DP-LW-semi-description}.
It follows from
$$
   |g(S'f_n)|\le|g|(|S'f_n|)\le|g|(S'|f_n|)\le|g|(T'|f_n|)\to 0  \ \ \ \ (\forall g\in E'')
$$
that $S'f_n\stackrel{\text{\rm w}}{\to}0$. 
Using Theorem \ref{DP-LW-semi-description} again, 
we get $S''\in\text{\rm DP-Lwc}(E'',F'')$.
\end{proof}
\noindent
The next lemma follows from Proposition \ref{prop elem}\,$ii)$
by Proposition \ref{l-LW-domination}.

\begin{lemma}\label{prop elem l-LW}{\em
Let an operator $T\in\text{\rm L}(E,F)$ possess the modulus. Then
\begin{enumerate}[$i)$]
\item
$T\in\text{\rm r-l-Lwc}(E,F)$ iff $|T|\in\text{\rm l-Lwc}(E,F)$.
\item
$T\in\text{\rm r-l-Mwc}(E,F)$ iff $|T|\in\text{\rm l-Mwc}(E,F)$.
\end{enumerate}}
\end{lemma}
\noindent
We conclude the paper with the following result.

\begin{theorem}\label{MW--LW--algebras}
The following statements hold:
\begin{enumerate}[$i)$]
\item 
$\text{\rm r-l-Lwc}(E)$ $($resp., $\text{\rm r-l-Mwc}(E)$$)$ is a subalgebra of $\text{\rm L}_r(E)$. 
\end{enumerate}
Moreover, 
$$
   \text{\rm r-l-Lwc}(E)=\text{\rm L}_r(E) \ \Longleftrightarrow I_E\in\text{\rm l-Lwc}(E),
   \eqno(5)
$$
$$
   \text{\rm r-l-Mwc}(E)=\text{\rm L}_r(E) \ \Longleftrightarrow I_E\in\text{\rm l-Mwc}(E).
   \eqno(6)
$$ 
\begin{enumerate}[$ii)$]
\item 
If $E$ is Dedekind complete then $\text{\rm r-l-Lwc}(E)$ $($resp., $\text{\rm r-l-Mwc}(E)$$)$ is 
a closed order ideal of the Banach lattice $(\text{\rm L}_r(E), \ \|\cdot\|_r)$. 
\end{enumerate}
\end{theorem}

\begin{proof}
We restrict ourselves to $\text{\rm r-l-Lwc}(E)$.
The case of $\text{\rm r-l-Mwc}(E)$ is analogues.

$i)$\ \
It follows from Corollary \ref{l-LW-algebra}
that $\text{\rm r-l-Lwc}(E)$ is a right ideal and hence is 
a subalgebra of $\text{\rm L}_r(E)$. 
The condition on $I_E$ under that $\text{\rm r-l-Lwc}(E)=\text{\rm L}_r(E)$ is trivial.

\medskip
Formulas (5) and (6) follow from Corollary \ref{l-LW-algebra}.

\medskip
$ii)$\ \ 
It follows from Lemma \ref{prop elem l-LW} that $\text{\rm r-l-Lwc}(E)$ is 
a Riesz subalgebra of $\text{\rm L}_r(E)$. Proposition \ref{vect lat} implies 
that $\text{\rm r-l-Lwc}(E)$ is an order ideal of $\text{\rm L}_r(E)$.
Let $S,T\in\text{\rm r-l-Lwc}(E)$ satisfy $|S|\le|T|$.
In order to show $\|S\|_r\le\|T\|_r$,
observe that $|S|,|T|\in\text{\rm l-Lwc}(E)$ by Lemma \ref{prop elem l-LW}. Then 
$
   \|S\|_r=\|~|S|~\|\le\|~|T|~\|=\|T\|_r.
$
Due to Lemma \ref{P-norm}, Lemma \ref{l-LW-closed} implies 
that $(\text{\rm r-l-Lwc}(E), \ \|\cdot\|_r)$ is a Banach space.
Clearly, the regular norm $\|\cdot\|_r$ is submultiplicative. Indeed,
if $S,T\in{\text{\rm L}}(E)$ then $|ST|\le|S|\cdot|T|$, and hence 
$
   \|ST\|_r\le\|~|S||T|~\|\le\|~|S|~\|\cdot\|~|T|~\|=
   \|S\|_r\cdot\|T\|_r,
$
as desired.
\end{proof}


{\small 
}

\end{document}